\documentclass[a4paper,12pt]{article}

\usepackage[latin1]{inputenc}
\usepackage[T1]{fontenc}
\usepackage[english]{babel}

\usepackage{mathrsfs}
\usepackage{amscd}
\usepackage{amsfonts}
\usepackage{amsmath}
\usepackage{amssymb}
\usepackage{amstext}
\usepackage{amsthm}
\usepackage{amsbsy}

\usepackage{xspace}
\usepackage[all]{xy}
\usepackage{graphicx}
\usepackage{url}
\usepackage{latexsym}

\usepackage{graphicx} % support the \includegraphics command and options

\usepackage{booktabs} % for much better looking tables
\usepackage{array} % for better arrays (eg matrices) in maths
\usepackage{paralist} % very flexible & customisable lists (eg. enumerate/itemize, etc.)
\usepackage{verbatim} % adds environment for commenting out blocks of text & for better verbatim
\usepackage{subfig} % make it possible to include more than one captioned figure/table in a single float
\usepackage{fancyhdr} % This should be set AFTER setting up the page geometry
\usepackage{sectsty}
\allsectionsfont{\sffamily\mdseries\upshape} % (See the fntguide.pdf for font help)
\usepackage[nottoc,notlof,notlot]{tocbibind} % Put the bibliography in the ToC
\usepackage[titles,subfigure]{tocloft} % Alter the style of the Table of Contents

\usepackage[textwidth=100pt,textsize=footnotesize,bordercolor=white,color=blue!30]{todonotes}
\usepackage{hyperref} 

\makeatletter
\newcommand*{\rom}[1]{\expandafter\@slowromancap\romannumeral #1@}
\makeatother

\theoremstyle{definition}

\newtheorem{fact}{fact}

\newtheorem{thm}[fact]{Theorem}

\newtheorem{corollary}[fact]{Corollary}
\newtheorem{defini}[fact]{Definition}

\title{A Note on the Decidability of the Necessity of Axioms}
\author{Merlin Carl}

\begin{document}

\maketitle

\begin{abstract}
A typical kind of question in mathematical logic is that for the necessity of a certain axiom: Given a proof of some statement $\phi$ in some axiomatic system $T$, one looks for minimal subsystems of $T$ that allow 
deriving $\phi$. In particular, one asks whether, given some system $T+\psi$, $T$ alone suffices to prove $\phi$. We show that this problem is undecidable unless $T+\neg\psi$ is decidable.
\end{abstract}

There are various places in mathematical logic where one is concerned with the necessity of certain axioms for the proof of a theorem; in set theory, typical questions are about
the necessity of the axiom of choice and large cardinal assumptions. In arithmetic, one is interested in the minimal degree of induction necessary for the proof of some statement.
Generally, it is hard to determine the answer. This suggests that these problems may be undecidable. We show that this is indeed the case for all the cases mentioned and in fact many more.
This is in fact an easy consequence of a result of Ehrenfeucht and Mycielski (see below). To the best of our knowledge, however, this consequence has so far not been noted or written
down. The purpose of this note is to change this.

\begin{defini}
The complexity of a proof is the number of symbols (quantifiers, junctors, variables, constants, relation and function symbols) occuring in it. If $B$ is a proof, then $|B|$ denotes its complexity.
If $T$ is a theory and $\phi$ is a provable statement of $T$, then $W_{T}(\phi)$ denotes $\text{min}\{|B|:B\text{ is a $T$-proof of }\phi\}$.
\end{defini}

The main ingredient of our proof is the following result of Ehrenfeucht and Mycielski (\cite{EM}):

\begin{thm}{\label{EM}}
If $T+\neg\alpha$ is undecidable, then there is no recursive function $f$ such that $W_{T}(\phi)\leq f(W_{T+\alpha}(\phi))$ holds for all theorems $\phi$ of $T$.
\end{thm}

\begin{thm}{\label{decidingweakertheories}}
Let $T$ be a first-order theory and $\phi$ a statement such that $T+\neg\phi$ is undecidable. Then there is no effective procedure to decide whether, given a proof in $T+\phi$ of some statement $\psi$,
the statement $\psi$ is provable in $T$.
\end{thm}
\begin{proof}
Assume for a contradiction that $P$ is a program deciding this question. Let $Q$ be a program proceeding as follows: Given a natural number $n$, compute the set $S_{n}$ of all
$T+\phi$-proofs of complexity at most $n$. Note that $S_{n}$ is finite and a code $c_{n}$ for $S_{n}$ can be computed uniformly in $n$. Using $P$, compute (a code for)
the subset $T_{n}\subseteq S_{n}$ of all elements of $S_{n}$ which prove a statement that is also provable in $T-\{\phi\}$. Using $T_{n}$, find (by exhaustively searching through the $T-\{\phi\}$-proofs)
a (finite) set $\mathbb{B}_{n}$ containing a $T-\{\phi\}$-proof for each element of $T_{n}$. Finally, output $\text{max}\{|B|:B\in\mathbb{B}_{n}\}$. Thus $Q$ computes, given $n$,
 an upper bound for the complexity of a $T-\{\phi\}$-proof
of a statement that has a $T$-proof of complexity $n$.\\
Now let $\psi$ be a statement provable in $T-\{\phi\}$, let $B$ be a $T$-proof of $\psi$ and let $Q(|B|)\downarrow=k$. 
Then, by construction of $Q$, there is a $T-\{\phi\}$-proof of $\psi$ of complexity at most $k$,
contradicting Theorem \ref{EM}.
\end{proof}

We note some particularly interesting special cases.

\begin{corollary}
\begin{enumerate}
 \item There is no effective procedure to decide whether a $ZFC$-theorem is provable in $ZF$ alone.
 \item Assuming the consistency of some large cardinal hypothesis $H$, there is no effective procedure to decide whether a $ZFC+H$-theorem is provable in $ZFC$ alone
 \item There is no effective procedure that maps $PA$-theorems $\phi$ to the degree of induction necessary for their proof, i.e. the smallest $n$ such that $I\Sigma_{n}\vdash\phi$
%To formulate this result, we pick an effective enumeration $(\phi_{i}|i\in\omega)$ of the $T$-theorems. (Such an enumeration can e.g. be obtained
%by letting $\phi_{n}$ be the statement proved by the $n$th proof in some effective enumeration of the $T$-proofs.)
\end{enumerate}
\end{corollary}
\begin{proof}
For (1) and (2), this follows from the observation that every consistent recursive extension of $ZF$ is undecidable.\\
Concerning (3), assume for a contradiction that $P$ is a program that, given a (code for) a $PA$-theorem $\psi$, outputs the smallest $n$ such that $\psi$ is provable in $I\Sigma_{n}$ (and does not halt when
the input is not provable in $PA$). We use the definability of bounded truth predicates in arithmetic to write the induction axioms for $\Sigma_n$ formulas as a single formula $\phi_{n}$ for each $n\in\omega$.
Now, every recursive consistent extension of $I\Sigma_{1}$ (i.e. $PA^{-}+\phi_{1}$) is undecidable. Since $\phi_{2}$ is not implied by $I\Sigma_{1}$, so $PA^{-}+\phi_{1}+\neg\phi_{2}$
is consistent and hence undecidable. But, as $I\Sigma_{2}$-theorems are $PA$-theorems, $P$ easily allows us to decide whether some theorem $\psi$ of $I\Sigma_{2}$ is provable in $I\Sigma_{1}$, a contradiction to Theorem \ref{decidingweakertheories}.
\end{proof}


\begin{thebibliography}{}
\bibitem[EM]{EM} A. Ehrenfeucht, J. Mycielski. Abbreviating proofs by adding new axioms. Bulletin of the American Mathematical Society, 77, pp. 366-367
\end{thebibliography}
\end{document}